\documentclass{article}
\usepackage{latexsym}
\usepackage[normalem]{ulem}
\usepackage[utf8]{inputenc} 
\usepackage[T1]{fontenc} 
\usepackage{booktabs} 
\usepackage{amsmath, amssymb, amsthm}  
\usepackage[top=1in, bottom=1in, left=1in, right=1in]{geometry}

\usepackage[outline]{contour}
\usepackage{tikz}      
\usetikzlibrary{shapes,automata,positioning,arrows,fit} 
\usetikzlibrary{decorations.pathmorphing}
\usepackage{tikz-qtree}

\newtheorem{theorem}{Theorem}[section]

\newtheorem{proposition}[theorem]{Proposition}
\newtheorem{lemma}[theorem]{Lemma}
\newtheorem{definition}[theorem]{Definition}
\newtheorem{example}[theorem]{Example}

\newtheorem{remark}[theorem]{Remark}

\newcommand{\been}{\begin{enumerate}}
\newcommand{\enen}{\end{enumerate}}
\newcommand{\beit}{\begin{itemize}}
\newcommand{\enit}{\end{itemize}}
\def\bal#1\eal{\begin{align}#1\end{align}}
\def\bal*#1\eal*{\begin{align}#1\end{align}}

\def\bedf#1\endf{\begin{definition}#1\end{definition}}

\makeatletter
\newcommand{\uset}[3][0ex]{%
  \mathrel{\mathop{#3}\limits^{
    \vbox to#1{\kern+6.5\ex@
    \hbox{#2}\vss}}}}
\makeatother

\makeatletter
\newcommand{\oset}[3][0ex]{%
  \mathrel{\mathop{#3}\limits^{
    \vbox to#1{\kern-1\ex@
    \hbox{$#2$}\vss}}}}
\makeatother

\begin{document}

\title{Multistationarity in\\cyclic sequestration-transmutation networks}
\author{
   Gheorghe Craciun \thanks{Department of Mathematics, University of
  Wisconsin-Madison, USA. {\tt craciun@math.wisc.edu}. Grant
  support from NSF DMS-1816238 and NSF DMS-2051568.}
  \and
Badal Joshi \thanks{Department of Mathematics, California State University San Marcos, USA. {\tt bjoshi@csusm.edu}.}
\and
Casian Pantea \thanks{Department of Mathematics, West Virginia University, USA. {\tt cpantea@math.wvu.edu}.}
\and
Ike Tan \thanks{Department of Mathematics, University of Michigan, USA. {\tt irtan@umich.edu}.}
  }

\maketitle


\begin{abstract}
We consider a natural class of reaction networks which consist of reactions where either two species can inactivate each other (i.e., sequestration), or some species can be transformed into another (i.e., transmutation), in  a way that gives rise to a feedback cycle. We completely characterize the capacity of multistationarity of these networks. This is especially interesting because such networks provide simple examples of ``atoms of multistationarity'', i.e., minimal networks that can give rise to multiple positive steady states.
\end{abstract}


\section{Introduction}

An important problem in the theory of reaction networks is to identify the networks that allow multiple (stoichiometrically compatible) positive equilibria~\cite{ Banaji.2009aa, Banaji.2007aa,  banaji2010graph, banaji2016some, conradi2017graph, craciun2005multiple, craciun2006multiple, joshi2012simplifying, joshi2013atoms, mincheva2007graph, wiuf2013power, yu2018mathematical}. This phenomenon, also referred to as {\em multistationarity}, underlies switching behavior in biochemistry~\cite{angeli2004detection, craciun2006understanding, ferrell2002self}. In particular, multistationarity is necessary in order for a reaction system to be able to generate multiple outputs in response to different external signals or stimuli. One prominent approach for identifying multistationary networks, developed over the last decade, is that of {\em network inheritance}, which says that multistationarity in a large, complex network can be established via studying smaller component networks that are also multistationary. Precise conditions for inheritance of multistationarity have been established~\cite{banaji2018inheritance, joshi2013atoms}. This creates the possibility of ``lifting'' multistationarity from certain idealized network motifs to larger and more realistic networks. Classes of motifs that have been catalogued by their presence or absence of multistationarity include all open networks with one (reversible or irreversible) reaction with arbitrary stoichiometry \cite{joshi2013complete}, bimolecular open networks with two reactions (both reactions reversible or irreversible) \cite{joshi2013atoms}, fully open as well as  isolated sequestration networks in arbitrary number of species  and reactions \cite{joshi2015survey}, and certain small non-open reaction networks, notably those with two species and two reactions \cite{joshi2017small}. In this paper, we add a new class of motifs to this catalog: cyclic sequestration-transmutation (CST) networks. We establish precise conditions for when this class of networks admits multistationarity. Our results contribute to the theoretical understanding of multistationarity in sparse networks, but they also have practical consequences. For example, in Section \ref{sec:Appl} we use inheritance of multistationarity and our results on CST networks to prove multistationarity of a well-known VEGFR dimerization model \cite{VEGF}.

Furthermore, our analysis is not limited to the fully open case as in many previous studies, but extends to partially open and non-open or isolated versions of CST networks. Moreover, the kinetics that we consider are not only mass action kinetics, but  includes a much broader class. These general kinetics must satisfy only mild requirements, such as: in order for a reaction to take place all reactants must have positive concentration, and the rate of reaction must increase if the concentration of a reactant increases.

\section{Background and notation}
We introduce terminology and recall some important results that will be useful in our proofs below. 

\subsection{Reaction networks and kinetics}

The general form of a {\em reaction} is 
\begin{equation}\label{eq:genReact}
a_1X_1+a_2X_2+\ldots+a_nX_n\to b_1X_1+b_2X_2+\ldots+b_nX_n,
\end{equation}
where $X_1,\ldots X_n$ is a list of {\em species}, and $a=[a_1,\ldots, a_n]^T$ and $b=[b_1,\ldots, b_n]^T$ are nonnegative integer-valued vectors whose entries are called {\em stoichiometric coefficients}. Formal linear combinations of $X_1,\ldots X_n$ are called {\em complexes}; in particular 
$a\cdot X := a_1X_1+a_2X_2+\ldots+a_nX_n$ is called the {\em source complex} of (\ref{eq:genReact}). Species $X_i$ for which $a_i>0$ are called {\em reactant species}. Finally, $a$ and $b-a$ are called the {\em source vector}, respectively {\em reaction vector} of (\ref{eq:genReact}).

\smallskip

Given a list of (distinct) species $X=(X_1,\ldots, X_n)$, a {\em reaction network} is a finite list of reactions on $X_1,\ldots, X_n$. It is customary to impose that the source and product complexes differ for each reaction, that every species participates in at least one reaction, and that no reaction is listed multiple times. However, none of these assumptions is needed in this paper. 

For a reaction network with $m$ reactions and a fixed ordering of the reactions, the {\em source matrix}, or {\em left stoichiometric matrix} $\Gamma_l\in{\mathbb R}^{n\times m}$ has the $m$ source vectors as columns. Likewise the {\em stoichiometric matrix} $\Gamma$ is the $n\times m$ matrix whose columns are the $m$ reaction vectors of the reaction list. The image of the stoichiometric matrix is called the {\em stoichiometric subspace} of the network.

The vector of {\em concentrations} of $X_1,\ldots, X_n$ is denoted by $x\in\mathbb R^n_{\ge 0}$\footnote{Throughout this article we use the convention that species names are in the upper case while their concentrations are the corresponding lower case letter. For example, the concentrations of species $X_1,\ldots, X_n$,  are denoted $x_1,\ldots, x_n$. }.

We ascribe a rate to each reaction, an assignment that is referred to as {\em kinetics}. Under {\em mass action kinetics}, reaction rates are proportional to the product of the concentrations of reactants (taken with multiplicity). To be precise, the rate of reaction $a_1 X_1 + \ldots + a_n X_n \xrightarrow{k} b_1 X_1 + \ldots + b_n X_n$ is  $kx^a=kx_1^{a_1} \cdots x_n^{a_n}$.
Here $k$ is a positive constant that depends on the reaction, called {\em reaction rate constant}. For mass-action kinetics it is customary to indicate the reaction rate on top of the reaction arrow. 

Some of the results to follow hold under {\em general kinetics} \cite[Definition 4.5]{banaji2016some}, a large class of reaction rates that includes mass-action kinetics as a special case. General kinetics places minimal physical requirements on reaction rates, like ``concentrations do not become negative'', ``reactions proceed if and only if all reactants are present'', and ``reaction rates are nondecreasing with reaction concentration''. To be exact, $v$ is a {\em general kinetics} rate vector if %
\begin{enumerate}
\item $v$ is defined and $C^1$ on $\mathbb{R}^n_{\geq 0}$;
\item $v_j \geq 0$; $v_j(x) = 0$ if and only if $x_i=0$ for some reactant $X_i$ of reaction $j$;
\item $\partial v_j/\partial x_i=0$ if $X_i$ is not a reactant in reaction $j$; $\partial v_j/\partial x_i$ is non-negative on $\mathbb{R}^n_{\geq 0}$ and strictly positive on the interior of the positive orthant if $X_i$ is a reactant in reaction $j$.
\end{enumerate}

In deterministic spatially homogeneous models, $x$ varies with time  according to the ODE system 
\begin{equation}\label{eq:genKin}
\dot x = \Gamma v(x)
\end{equation}
where $v(x)=(v_1(x),\ldots v_m(x))$ is the vector of reaction rates, or {\em rate vector} (see Example~\ref{ex:odes} on the next page). 

A reaction of the form $X_i\to 0$ (by which $X_i$ is depleted or degraded) is called an {\em outflow reaction}; its reaction vector  $[0,\ldots, 0, -1,0,\ldots, 0]^T$ has a single nonzero entry at index $i$.
With mass-action kinetics the outflow reaction $X_i \to 0$  has rate $k x_i$. Likewise the {\em inflow reaction} $0 \to X_i$ has  reaction vector  $[0,\ldots, 0, 1,0,\ldots, 0]^T$, and constant rate under mass-action. Inflow and outflow reactions are referred to as {\em flow reactions}.

\begin{definition}[Open, fully open and closed networks]\label{def:openNetw}
A network which contains at least one flow reaction is called {\em open}. A network which contains inflow and outflow reactions $0\to X_i$ and $X_i \to 0$  for each species $X_i$ is called {\em fully open}. A network that is not open is called {\em closed}. 
\end{definition}

\begin{remark}
While ``fully open network''  is standard terminology in reaction networks, the meaning of ``open'' and ``closed'' network may vary in the literature. Our notion of ``open network'' is the same as that of \cite{craciun2006multiple}. Reactions of our ``closed'' networks are sometime referred to as ``true'' reactions \cite{craciun2005multiple}. 
\end{remark}

\begin{remark}\label{rem:block}
A reaction network in $n$ species which contains $R$ outflow reactions and $r_2$ inflow reactions has stoichiometric and reactant  matrices written in block form
$$\overline\Gamma=[\Gamma\vert -K\vert L] \text{ and }
\overline\Gamma_l=[\Gamma_l\vert K\vert 0]$$
where $K\in{\mathbb R}^{n\times R}$ and $L\in {\mathbb R}^{n\times r_2}$ are submatrices of the identity $I_n$.
If the network is fully open then $K=L=I_n.$

As above, throughout the paper we will denote by $\Gamma$ and $\Gamma_l$ the source and stoichiometric matrices of closed network, whereas $\overline\Gamma$ and $\overline\Gamma_l$ will be used for open networks. 
\end{remark}

\subsection{Compatibility classes and multistationarity}

Integrating \eqref{eq:genKin} with respect to time yields
$$x(t)=x(0)+\Gamma\int_{0}^T v(x(s))\text{d}s.$$
Under general kinetics, $x(t)$ is nonnegative for any $t\ge 0$, and so the solutions of \eqref{eq:genKin} are constrained to {\em compatibility classes}, i.e. sets of the form  $(x_0+\text{im}(\Gamma))\cap{\mathbb R}^n_{\ge 0}$, where $x_0\in{\mathbb R}_{\ge 0}^n$. 

Let $\cal R$ denote a reaction network with stoichiometric matrix $\Gamma$ and fix a general kinetics $v$. A {\em positive steady state} of $\cal R$ is a point $x^*\in\mathbb R^n_{>0}$ such that $$\Gamma v(x^*)=0.$$ 
A steady state $x^*$ is called {\em nondegenerate} if the {\em reduced Jacobian}, (i.e. the Jacobian  of the vector field $\Gamma v(x)$  restricted to the compatibility class of $x^*$) is nonzero \cite{banaji2016some}. Equivalently, if $r=\text{rank }\Gamma$ then  $x^*$ is nondegenerate if the the sum of the $r\times r$ principal minors of the Jacobian matrix $\Gamma Dv$ computed at $x^*$ is nonzero. We note that if the stoichiometric matrix $\Gamma\in {\mathbb R}^{n\times m}$ of $\cal R$ has full column rank $n$ then the reduced Jacobian coincide with the Jacobian of the vector field. We remark that the Jacobian matrix of a mass-action reaction network vector field has the convenient form 

\begin{equation}\label{eq:Jac}\left(\frac{\partial \Gamma v}{\partial x}\right)=\Gamma D_{v(x)}\Gamma_l^T D_{1/x},
\end{equation}
where $D_{v(x)}$ and $D_{1/x}$ are diagonal matrices with $(v_1(x),\ldots, v_m(x))$ and $(1/x_1,\ldots, 1/x_n)$ on the diagonals. Note that in this formula the inflow reactions can be excluded from $\Gamma$ and $\Gamma_l$ without changing the result.

\begin{definition}[multistationarity, nondegenerate multistationarity] Let $\cal R$ denote a reaction network with stoichiometric matrix $\Gamma$. 
\begin{enumerate}
\item ${\cal R}$ is {\em (nondegenerately) multistationary under mass-action kinetics} if there exists a choice of mass action kinetics $v$ (i.e. a choice of rate constants) such that \eqref{eq:genKin} has two distinct (nondegenerate) positive steady states within the same compatibility class.
\item ${\cal R}$ is {\em (nondegenerately) multistationary under general kinetics} if there exists a choice of general kinetics $v$ such that \eqref{eq:genKin} has two (nondegenerate) distinct positive steady states within the same compatibility class.
\end{enumerate}
\end{definition}

\begin{example}\label{ex:odes}
Consider for example the reaction network
\begin{equation}
\label{eq:ex3}
2X_1+X_2\overset{k_1}{\to} 3X_1,\  X_1\overset{k_2}{\to} X_2, 
\end{equation}
one of the simplest networks with multistationarity (see also \cite{joshi2013atoms, banaji2018inheritance}).

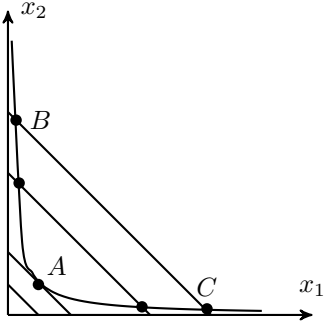
\begin{figure}
\begin{center}
\begin{tikzpicture}[scale=1.5, thick, >=stealth', scale = 1.8]
\draw[->] (0,0) -- (1.5,0);
\draw[->] (0,0) -- (0,1.5);
\draw[scale=0.5,domain=0.037:2.5,smooth,variable=\y]  plot ({\y},{.1/\y});
\draw (0,.15) -- (.15,0);
\draw (0,.31) -- (.31,0);
\draw (0,.7) -- (.7,0);
\draw (0,1) -- (1,0);
\fill (.15,.15) circle (.8pt);
\fill (.055,.65) circle (.8pt);
\fill (.04,.96) circle (.8pt);
\fill (.66,.04) circle (.8pt);
\fill (.98,.03) circle (.8pt);
\node (a) at (.24,.24) {$A$};
\node (a) at (.16,.96) {$B$};
\node (a) at (.98,.14) {$C$};
\node (a) at (1.5,.13) {$x_1$};
\node (a) at (.13, 1.5) {$x_2$};
\end{tikzpicture}
\end{center}
\caption{Multistationarity of reaction network ({\ref{eq:ex3}}). Set $k_1=k_2=1$. The compatibility class $x_1 + x_2=T$ contains no positive steady states for $T<2$, two nondegenerate steady states for $T>2$, and a single degenerate equilibrium for $T=2$.}
\label{fig:ex3}
\end{figure}

The stoichiometric matrix and source matrix of the network are
$$
\Gamma=
\begin{bmatrix}
1 &-1\\
-1 &1
\end{bmatrix},
\quad 
\Gamma_l=
\begin{bmatrix}
2 &1\\
1 &0
\end{bmatrix}
$$
and the steady state manifold $k_1x_1^2x_2-k_2x_1=0$ intersects the positive quadrant along the curve $x_1x_2=k_2/k_1$. The positive compatibility classes are obtained by intersecting the positive quadrant with cosets of $\text{span}([1,-1]^T)$  i.e. lines of the form $x_1 + x_2=T$. Compatibility classes may contain no positive steady states, a single degenerate steady state, or two nondegenerate steady states (see Figure \ref{fig:ex3}). 
\end{example}

To be a little more precise, one computes the Jacobian matrix of the system as 
$$
\Gamma Dv=
\begin{bmatrix}
1 &-1\\
-1 &1
\end{bmatrix}
\begin{bmatrix}
2k_1x_1x_2 &k_1x_1^2\\
k_2 &0
\end{bmatrix}=
\begin{bmatrix}
2k_1x_1x_2-k_2 &k_1x_1^2\\
-2k_1x_1x_2+k_2 &-k_1x_1^2
\end{bmatrix}
$$
For simplicity set $k_1=k_2=1$ (but the same calculation can be done for any choice of $k_1$ and $k_2$). Since $\text{rank }\Gamma=1,$ the reduced Jacobian at a steady state equals $\text{Tr}(\Gamma Dv)=2x_1x_2-x_1^2-1=1-x_1^2$. It follows that the only degenerate steady state is $(1,1).$

\subsection{Ruling out multistationarity: the injectivity property}

A particularly successful approach in the study of multistationarity has been that of {\em injective} reaction networks, i.e. reaction networks  for which the corresponding vector field  is injective  on each compatibility class, for any choice of kinetics.

\begin{definition}[injective reaction networks] Let $\cal R$ denote a reaction network with stoichiometric matrix $\Gamma$. 
\begin{enumerate}
\item ${\cal R}$ is {\em injective under mass-action} if for any choice of mass-action kinetics $v(x)$, the restriction of the vector field $f(x)=\Gamma v(x)$ to any positive compatibility class is injective.

\item ${\cal R}$ is {\em injective under general kinetics} if for any choice of general kinetics $v(x)$, the restriction of the vector field $f(x)=\Gamma v(x)$ to any positive compatibility class is injective.
\end{enumerate}
\end{definition}

Since mass action kinetics are a subclass of general kinetics, it is clear that a network that is injective under general kinetics is injective under mass action kinetics. 
However, a network may fail to be injective under general kinetics while being injective under mass action kinetics \cite{banaji2016some}.

An injective reaction network cannot be multistationary, since  requires that two different points in the same positive compatibility class be both mapped by $f$ to zero. Note however that injectivity is  not equivalent to the lack of capacity for multiple positive equilibria \cite{craciun2005multiple}.

The study of injective reaction networks was started by Craciun and Feinberg for fully open networks \cite{craciun2005multiple} and has since been extended  by work of many authors \cite{banaji2016some, Mueller.2016aa, joshi2012simplifying, feliu2013simplifying, wiuf2013power,  craciun2006multiple, Shinar.2012aa, Banaji.2007aa}.  The characterization of injectivity we give in Theorem \ref{thm:inj} is that of \cite[Theorems 3, 5]{banaji2018inheritance}.  For different versions of this result and  for other related results the reader is referred to  \cite{joshi2012simplifying,Mueller.2016aa,Shinar.2012aa}.

For a matrix $A\in{\mathbb R}^{n\times m}$ and two sets $\alpha\subseteq \{1,\ldots, n\}, \beta\subseteq \{1,\ldots, m\}$ of the same cardinality, we denote by $A[\alpha\vert\beta]$ the minor of $A$ corresponding to rows $\alpha$ and columns $\beta$. We also let ${\mathcal Q}(A)$ denote the set of real $n\times m$ matrices whose entries have the same sign ($+, -$ or 0) as the corresponding entry in $A$.

\begin{theorem}\label{thm:inj}
\begin{enumerate} Let $\cal R$ be a reaction network with stoichiometric matrix $\Gamma$ and reactant matrix $\Gamma_l$. Let $r=\text{rank } \Gamma.$
\item ${\cal R}$ is injective under mass-action if and only if
the products
$$\Gamma[\alpha\vert \beta]\Gamma_l[\alpha\vert \beta]$$
have the same sign for all choices of $\alpha\subseteq \{1,\ldots, n\}$ and $\beta\subseteq \{1,\ldots, m\}$ with $\vert\alpha\vert=\vert\beta\vert=r$, and at least one such product is nonzero.
\item $\cal R$ is injective under general kinetics if and only if
for any $A\in{\mathcal Q}(\Gamma_l)$
the products
$$\Gamma[\alpha\vert\beta]A[\alpha\vert\beta]$$
have the same sign for all choices of $\alpha\subseteq \{1,\ldots, n\}, \beta\subseteq \{1,\ldots, m\}$ with $\vert\alpha\vert=\vert\beta\vert=r$,  and at least one such product is nonzero.
\end{enumerate}
\end{theorem}

\subsection{The Jacobian optimization criterion}

The following sufficient condition for multistationarity  \cite[Theorem~4.1]{craciun2005multiple} will be used in some of our proofs.  Here we follow the formulation of this result given in \cite[Theorem~5]{banaji2016some}): 

\begin{theorem}\label{thm:craciunSuff} Consider a fully open reaction network $\overline{\cal R}$ and let $\overline\Gamma$ and $\overline\Gamma_l$ denote the  stoichiometric and source matrices corresponding to $\overline{\cal R}$ {\em with all inflow reactions omitted}. Suppose there exists a positive diagonal matrix $D$ such that
$$(-1)^n\det(\overline\Gamma D\overline\Gamma_l^T)<0 \mbox{ and }\overline\Gamma D{\bf 1}\le 0$$
where ${\bf 1}\in{\mathbb R}^{2n\times 1}$ denotes the vector of 1's. Then the fully open CST admits multiple positive equilibria under mass action.
\end{theorem}

\subsection{Inheritance of multistationarity} 

The following useful result states that nondegenerate multistationarity of a mass-action network survives when we add all possible flow reactions (see \cite[Theorem 2]{craciun2006multiple}, \cite[Corollary 3.6]{joshi2013atoms}, \cite[Theorem 2]{banaji2018inheritance}).

\begin{theorem}\label{thm:inheritance} ({\em Adding inflows and outflows of all species}). Let $\cal R$ denote a mass-action reaction network with species $X_1,\ldots, X_n$.
Suppose we create ${\cal R}'$
from $\cal R$ by adding to $\cal R$ all the flow reactions $0\rightleftharpoons X_1$, $\ldots,$ $0\rightleftharpoons X_n$. If $\cal R$ admits multiple positive nondegenerate steady states, then so does ${\cal R}'$.
\begin{enumerate}
\item (Adding a new species with inflow and outflow  \cite[Theorem 4.2]{joshi2013atoms} \cite[Theorem 4]{banaji2018inheritance}). Suppose we create ${\cal R}'$
from $\cal R$, by adding into some reactions of $\cal R$ the new species $Y$ in an arbitrary way, while also adding the new reaction $0\rightleftharpoons Y$ . If $\cal R$ is nondegenerately multistationary, then so is ${\cal R}'$.
\item (Adding a dependent reaction \cite[Theorem 3.1]{joshi2013atoms} \cite[Theorem 1]{banaji2018inheritance} ). Suppose we create ${\cal R}'$
from $\cal R$, by adding to $\cal R$ a
new irreversible reaction whose reaction vector lies in the stoichiometric subspace of $\cal R$. If $\cal R$ nondegenerately multistationary, then so is ${\cal R}'$. In particular, adding the reverse of any reaction in $\cal R$ preserves nondegenerate multistationarity. 
\end{enumerate}
\end{theorem}

We note that the result in Theorem \ref{thm:inheritance} is one of the simplest examples of modifications that preserve nondegenerate multistationarity. For a more complete list of such results the reader is referred to  \cite{banaji2018inheritance}.

\section{Cyclic sequestration-transmutation networks}\label{sec:seqnet}

We are interested in two special reaction types, sequestration and transmutation. Each involves exactly two distinct species. In the first case both appear on the reactant side, while in the latter case one appears on the reactant end and the other on the product end. 
The central object of study in this paper is the set of cyclic sequestration-transmutation networks. 

\begin{definition}{\em (CST networks)}\label{def:CST} \quad 
\begin{enumerate}
\item A {\em sequestration reaction} is a reaction of the type $aX + bY \to 0$ for positive integers $a,b$ and species $X,Y$.
\item A {\em transmutation reaction} is a reaction of the type $aX \to bY$ for positive integers $a,b$ and distinct species $X,Y$.
\item A reaction network on species $X_1,\ldots, X_n$ (with the convention $X_{n+1}=X_1$) containing reactions $R_1,\ldots, R_n$, where for each $i\in\{1,\ldots, n\}$ $R_i$ is either a sequestration  reaction
$$a_i X_i + b_{i+1} X_{i+1} \to 0$$
or a transmutation reaction
$$a_i X_i \to b_{i+1} X_{i+1},$$
is called a  {\em closed CST  (cyclic sequestration-transmutation) network}. If in addition, the network contains at least one flow reaction, then it is called an {\em open  CST network}. A {\em CST network} could mean either a closed or an open CST network. A {\em fully open CST network} contains inflow and outflows for all its species. 
\end{enumerate}
\end{definition}

\begin{remark} Subclasses of CST networks have been previously considered in the literature \cite{joshi2015survey}. 
Notably, the class of open {\em sequestration networks} 
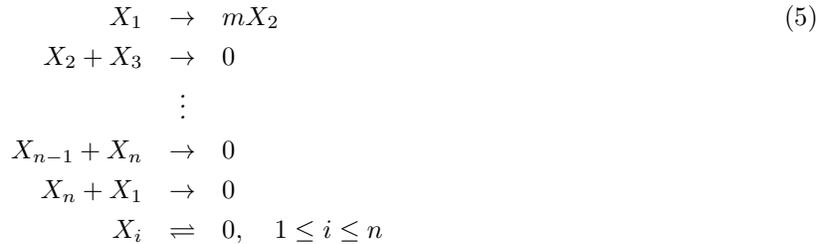
\begin{eqnarray}\label{net:seq}
X_1&\to& mX_2\\\nonumber
X_2+X_3&\to& 0\\\nonumber
&\vdots\\\nonumber
X_{n-1}+X_n&\to& 0\\\nonumber
X_n+X_1&\to& 0\\\nonumber
X_i&\rightleftharpoons& 0,\quad  1\le i\le n
\end{eqnarray}
($n\ge 2$, $m\ge 1$ integers) has been shown to be multistationary if and only if $m>1$ and $n$ is odd \cite[Theorem 6.4]{joshi2015survey}. Note that Theorems \ref{thm:CSToutflows} and \ref{thm:main} in this paper extend that result. Nondegeneracy of steady states in sequestration networks (\ref{net:seq}) has been shown for particular cases (including for network \eqref{net:seq} with $n=3$ in \cite{felix2016analyzing} and in full generality in \cite{tang2019bistability}. Furthermore, the latter work shows {\em bistability} (existence of multiple stable steady states) of sequestration networks with $m>1$ and $n$ odd. 
We note that nondegeneracy holds for classes of CST networks that are not necessarily sequestration networks; this is ongoing work \cite{DSRcyles}.
\end{remark}

\subsection{Closed CST network}
The stoichiometric matrix and the reactant matrix of a closed CST network can be written as 
\begin{equation}\label{eq:Gamma}
\Gamma=
\begin{bmatrix}
-a_1    &0 &\ldots &0 &-b_1\vert b_1\\
-b_2\vert b_2  &-a_2  &\ldots &0 &0\\
&\ddots &\ddots &\vdots &\vdots \\
&  &\ddots &-a_{n-1} &0 &\\
&  & &-b_n\vert b_n &-a_n\\
\end{bmatrix}
\text{ and }
\Gamma_l=
\begin{bmatrix}
a_1    &0 &\ldots &0 &b_1\vert 0\\
b_2\vert 0  &a_2  &\ldots &0 &0\\
&\ddots &\ddots &\vdots &\vdots \\
&  &\ddots &a_{n-1} &0 &\\
&  & &b_n\vert 0 &a_n\\
\end{bmatrix}
\end{equation}

The convention here is that the entry $-b_i\vert b_i$ in $\Gamma$ in is equal to $-b_i$ if $R_{i-1}$ is a sequestration reaction, and is equal to $b_i$ if $R_{i-1}$ is a transmutation reaction ($R_0$ is reaction $R_n$).

In the same way the entry $b_i\vert 0$ in $\Gamma_l$ is equal to $b_i$ if $R_{i-1}$ is a sequestration and to 0 if $R_{i-1}$ is a transmutation.

Under mass-action kinetics the rate of reaction $R_i$ is either
$$v_i(x)=k_ix_i^{a_i}x_{i+1}^{b_{i+1}} \text{\quad  or\quad }
v_i(x)=k_ix_i^{a_i}$$ 
depending on whether $R_i$ is a sequestration or transmutation reaction,  respectively.

The Jacobian matrix $Dv(x)$ of $v(x)$ has only nonnegative entries for general kinetics. Furthermore, an entry of $Dv(x)$ is positive if and only if the corresponding entry of $\Gamma_l$ is positive. In other words, $Dv(x)$ belongs to ${\cal Q}(\Gamma_l)$.

\begin{remark}
Unless we specify otherwise, we reserve $\Gamma, \Gamma_l$, and $v$ for denoting the stoichiometric matrix, reactant matrix and the rate vector of a closed CST. For open CSTs the stoichiometric and reactant matrices and the rate vector will be denoted by $\overline\Gamma$, $\overline\Gamma_l$, and $\overline v$ respectively.
\end{remark}

\subsection{Minors of $\Gamma$ and $\Gamma_l$ for closed CSTs}

Both $\Gamma$ and $\Gamma_l$ corresponding to closed CSTs belong to the class ${\cal M}\subset {\mathbb R}^{n \times n}$ of matrices of the form

\begin{equation}\label{eq:Mclass}
\left[
\begin{matrix}
m_{1,1}    &0       &\ldots &0        &m_{1,n}\\
m_{2,1}  &m_{2,2}     &\ldots &0        &0\\
0       &m_{3,2}  &\ldots &0        &0\\
\vdots       &\vdots       &\ddots &\vdots        &\vdots\\
0       &0       &\ldots &m_{n-1,n-1} &0\\
0       &0       &\ldots &m_{n,n-1}   &m_{n,n}\\
\end{matrix}
\right]
\end{equation}

\begin{lemma}\label{lem:CSTminors}
Let $M\in\cal M$. Any minor of $M$ of size less than $n$ is a monomial in $m_{i,j}$. The determinant of $M$ is equal to $\prod_{l=1}^n m_{l,l}+(-1)^{n+1}\prod_{l=1}^n m_{l,l-1}$, where by convention $m_{1,0}=m_{1,n}$.
\end{lemma}

\begin{proof}
Let $1\le i_1<\ldots<i_k\le n$, $1\le j_1<\ldots<j_k\le n$ and let $\alpha=\{i_1,\ldots,i_k\}$, $\beta=\{j_1,\ldots,j_k\}$.
Then

\begin{equation}\label{eq:minor}
M[\alpha\vert \beta]=\sum_{\sigma\in S_k} \epsilon(\sigma)\prod_{l=1}^k m_{i_l,j_{\sigma(l)}}
\end{equation}

For any $\sigma\in S_k$ that produces a non-zero term in (\ref{eq:minor}) we must have 
$j_{\sigma(l)}\in\{i_l-1,i_l\}$
where the index $0$ means $n$ by convention. 
This convention does not affect the case  $i_1>1$, where we have $j_{\sigma(1)}\le i_{1}\le j_{\sigma(2)}\le\ldots\le j_{\sigma(k)}\le i_{k}$. In this case 
$j_{\sigma (1)}<\ldots<j_{\sigma (k)}$,  $j_{\sigma(l)}=
j_l$ for all $l\in\{1,\ldots, k\}$, and therefore there is only one permutation that produces a nonzero term in  $M[\alpha\vert \beta]$ which is a monomial in $m_{i,j}$.

If $i_1=1$, then either $j_{\sigma(1)}  =1$ or $j_{\sigma(1)}=n$. If $j_{\sigma(1)}=1$ the argument above yields the permutation $\sigma_1(l)=l$ for all $l$. If $j_{\sigma(1)}=n$ then $n\in\alpha$ and therefore $i_k=n$. We have 
$$1=i_{1}\le i_{2}-1\le j_{\sigma(2)}\le i_2\le i_3-1\le j_{\sigma(3)}\ldots\le 
i_{k-1}\le i_{k}-1\le 
j_{\sigma(k)}\le i_{k}\le j_{\sigma(1)}=n$$ 
and we arrive at the permutation $\sigma_2(2)=1,\ \sigma_2(3)=2, \ldots \sigma_2(k)=k-1$, and $\sigma_2(1)=k$. Therefore there are at most two permutations $\sigma_1$ and $\sigma_2$ that produce nonzero terms, and note that $\sigma_1(l)\neq \sigma_2(l)$ for all $l$. Since $j_{\sigma(l)}\in\{i_l-1, i_l\}$, it follows that 
$i_l-1\in\alpha$ for all $l$. Therefore  $n-1\in\alpha, n-2\in\alpha,\ldots,1\in\alpha$. We have $\alpha=\beta=\{1,\ldots, n\}$ and in this case
$M[\alpha\vert \beta]=\det(M).$
\end{proof}

\begin{lemma}\label{lem:correspProd}
Let $\Gamma$ and $\Gamma_l$ denote the stoichiometric and source matrices of a closed CST network, and let $\alpha,\beta\subseteq \{1,\ldots, n\}$ be  such that $\vert \alpha\vert =\vert \beta\vert <n$. Then for any $M\in{\mathcal Q}(\Gamma_l)$ we have
$$(-1)^{\vert \alpha\vert }\Gamma[\alpha\vert \beta]M[\alpha\vert \beta]\geq 0.$$
\end{lemma}
\begin{proof}
Suppose $\Gamma[\alpha\vert \beta]M[\alpha\vert \beta]\neq 0$. By Lemma \ref{lem:CSTminors} the minors $
\Gamma[\alpha\vert \beta]$ and $M[\alpha\vert \beta]$ each contain one term, corresponding to the same permutation $\sigma$ in (\ref{eq:minor}). The sign of $M[\alpha\vert \beta]$ is equal to $\epsilon(\sigma)$. If the monomial $\Gamma_l[\alpha\vert \beta]$ contains $b_i$ as a factor, then $R_i$ is a sequestration, and the entry $-b_i\vert b_i$ in $\Gamma$ is equal to $-b_i$. Then $\Gamma[\alpha\vert \beta]=\epsilon(\sigma)\prod_{i\in I}(-a_i)\prod_{j\in J}(-b_j)=\epsilon(\sigma)(-1)^{\vert \alpha\vert }\prod_{i\in I}a_i\prod_{j\in J} b_j$, where $I, J$ are disjoint subsets of $\{1,\ldots, n\}$ with $\vert I\cup J\vert =\vert \alpha\vert $. It follows that sign($\Gamma[\alpha\vert \beta]M[\alpha\vert \beta])=(-1)^{\vert \alpha\vert }\epsilon(\sigma)^2$ and the conclusion follows.
\end{proof}

\section{Injectivity of CSTs}

\subsection{Closed CST networks}

We show that closed CST networks with mass-action are always injective. In fact, even under general kinetics, with the  exception of one case,  closed CST networks are injective. 
For the exception, which occurs when the number of species is even and transmutation reactions are absent, we can find a non-mass action kinetics that makes the closed CST network fail injectivity \cite{banaji2016some}.  

\begin{theorem}\label{thm:CSTinj}
\noindent 1. A closed CST network is not injective under general kinetics if and only if $s=n$, $n$ is even and $\prod_{i=1}^n a_i\neq\prod_{i=1}^n b_i$. 

\noindent 2. Any closed CST network is injective under mass-action.

\end{theorem}
\begin{proof}

1. We apply Theorem \ref{thm:inj} part 1 distinguishing  the cases  $\mbox{rank}\ \Gamma=n$ and $\mbox{rank}\ \Gamma\neq n$. Note that the second case is equivalent to $\mbox{rank}\ \Gamma= n-1$ since the minor $\Gamma[\{1,\ldots, n-1\}\vert \{1,\ldots, n-1\}]$ equals $(-1)^{n-1}\prod_{i=1}^{n-1}a_i\neq 0$. Let $M=Dv(x)\in\mathcal{Q}(\Gamma_l)$; $M$ is of the form (\ref{eq:Mclass}) with non-negative entries, and strictly positive diagonal entries. 

Suppose $\mbox{rank } \Gamma=n$. We have
$$\det\Gamma=(-1)^n\prod a_i+(-1)^{n+s+1}\prod b_i\neq 0,$$
i.e. $s$ is odd, or $\prod_{i=1}^n a_i\neq \prod_{i=1}^n b_i$. In this case injectivity is equivalent to $\det M\neq 0$. If $s<n$ then $\det M$ is the product of its diagonal entries, and is therefore strictly positive. If $s=n$ then
$$\det M = \prod_{i=1}^nm_{i,i}+(-1)^{n+1}\prod_{i=1}^n m_{i,i-1}$$ by Lemma \ref{lem:CSTminors}
and we have injectivity if and only if this expression is nonzero for any choice of positive $m_{i,j}$, i.e. if and only if $n$ is odd.
Therefore if $\mbox{rank } \Gamma=n$ the network is injective in all cases except when $s=n$ and $n$ is even. Note that in this case $\mbox{rank } \Gamma=n$ is equivalent to $\prod_{i=1}^n a_i\neq\prod_{i=1}^n b_i$.

In the remaining case where $\mbox{rank }\Gamma=n-1$ (i.e. $s$ is even and  $\prod a_i=\prod b_i$), Lemma  \ref{lem:correspProd} implies that if $\vert \alpha\vert =\vert \beta\vert =n-1$ then the sign of $\Gamma[\alpha\vert \beta]M[\alpha\vert \beta]$ is equal to $(-1)^{n-1}$. We also note that if $\alpha=\beta=\{1,\ldots, n-1\}$ this product is nonzero. This completes the verification of the hypothesis in Theorem \ref{thm:inj} part 2, and the network is injective.

2. We only need to discuss the case that fails injectivity under general kinetics, i.e. $s=n$ are even and $\prod_{i=1}^n a_i\neq \prod_{i=1}^n b_i$. In this case $\det \Gamma=\det\Gamma_l=\prod_{i=1}^n a_i-\prod_{i=1}^n b_i$ and so $\det\Gamma\det\Gamma_l>0$, which completes the proof using Theorem \ref{thm:inj} part 1.  
 \end{proof}

{\bf A note on stability of equilibrium points.} While stability and convergence to equilibria for CST networks is not the focus of this paper, we briefly note that monotone systems (\cite{Angeli.2010aa}; see also \cite{Craciun.2011ox}) and deficiency theory \cite{Horn.1972aa, Horn.1972ab, feinberg2019foundations, Craciun.2009aa, Pantea.2012ss, Craciun.2019df} are promising avenues for this type of question. In particular, the following  proposition on closed CST networks is easy to prove. Not to distract from the main focus of the paper, we will assume familiarity with the statement of Deficiency Zero Theorem and  related terminology; see \cite{feinberg2019foundations}.

\begin{proposition}
Consider a closed CST network denoted as in Definition \ref{def:CST} such that $s$ is odd or $\prod_{i=1}^n a_i\neq \prod_{i=1}^n b_i$ (i.e. $\Gamma$ has rank $n$), and let ${\cal R}$ be obtained from this closed CST network by making all reactions reversible. Then $\cal R$ equipped with mass-action has a unique positive equilibrium, which is locally asymptotically stable. 
\end{proposition}
\begin{proof}
The claim follows from the Deficiency Zero Theorem \cite[Section 7.1]{feinberg2019foundations}. There is a linkage class which is a graph-theoretical star that contains all sequestrations and no other reactions. Transmutation reactions form $k$ linkage classes which are graph-theoretical paths (if one of these formed a cycle, then we would obtain linear dependencies between reaction vectors, and $\mathrm{rank}(\Gamma)<n$).  
Suppose there are $k$  linkage classes of transmutations. The number of complexes in transmutation reactions is $t+k$, and we compute the deficiency of the CST network as follows: $(s+1+t+k)-(1+k)-n=0$.
\end{proof}

\begin{remark}
If $\prod_{i=1}^n a_i=\prod_{i=1}^n b_i$ then the fully open CST (without reverse reactions) has a unique positive steady state, which is linearly stable. This follows since the fully open CST is {\em delay stable}, i.e. when modeled as a mass-action system with delays, the steady state is linearly stable for any choice of the delay parameters \cite[Example 5.10]{delayPolly}.
\end{remark}

\subsection{Open CSTs without outflows}
The proof of Theorem \ref{thm:CSTinj} carries over without additional effort if we add inflows for some arbitrary subset of species (but no outflows). Indeed, if the stoichiometric matrix and left stoichiometric matrix of the CST (without inflows) are denoted $\Gamma$ and $\Gamma_l$, then
$\overline\Gamma=[\Gamma\vert L]$ and $\overline\Gamma_l=[\Gamma_l\vert 0]$ are the stoichiometric and left stoichiometric matrices of the open CST with inflows (see Remark \ref{rem:block}). Then  $\mbox{rank } \overline\Gamma=n$, and the only  non-zero product $\overline\Gamma[\alpha\vert \beta]\overline\Gamma_l[\alpha\vert \beta]$ corresponds to $\alpha=\beta=\{1,\ldots, n\}$. It follows from  Theorem \ref{thm:inj} part 2 that the CST network with inflows is injective. We have the following

\begin{theorem}\label{thm:CSTinjopen}
\noindent  \noindent 1. An open  CST network without outflows is not injective under general kinetics if and only if $s=n$, $n$ is even and $\prod_{i=1}^n a_i\neq\prod_{i=1}^n b_i$.

\noindent 2. Any open CST network without outflows is injective under mass-action.

\noindent \end{theorem}

\subsection{Open CSTs with outflows}

\begin{theorem}\label{thm:CSToutflows}
An open CST with $s$ sequestration reactions which contains at least one outflow reaction is injective under mass action if and only if $n=s$, or $s$ is odd, or $\prod_{i=1}^n a_i\ge \prod_{i=1}^n b_i$.
\end{theorem}

\begin{proof}
The stoichiometric and left stoichiometric matrices are in this case $\overline\Gamma=[\Gamma\vert \!\!-\!\!K\vert L]$ and $\overline\Gamma_l=[\Gamma_l\vert K\vert 0]$, where $K\in{\mathbb R}^{n\times k}$ is a submatrix of the identity matrix. We note that $\overline\Gamma$ has rank $n$. 
To apply Theorem \ref{thm:inj}b, we first show that at least one product of minors 
\begin{equation}\label{eq:prodMinors}
\overline\Gamma[\{1,\ldots,n\}\vert \beta]\overline\Gamma_l[\{1,\ldots,n\}\vert \beta]
\end{equation}
is nonzero \mbox{ for some index set } $\beta$ \mbox{ with } $\vert \beta\vert =n.$
Indeed, if  species $X_j$ is in the outflow, then consider the minor of $\overline\Gamma$ that corresponds to the column $[0,\ldots, 0,-1,0,\ldots, 0]^T$ (-1 at position $j$) and all columns of $\Gamma$  except column $j$. Consider also the corresponding minor of $\overline\Gamma_l$. The product of the two minors is equal to

$$
\begin{vmatrix}
-a_1  & & & & & &-b_1\vert b_1\\
-b_2\vert b_2  &-a_2  & & & & &\\
&\ddots &\ddots & & & &\\
&  &\ddots &-1 & & & \\
& & &0 &\ddots & & \\
& & &  &\ddots &-a_{n-1} & \\
& & &  & &-b_n\vert b_n &-a_n\\
\end{vmatrix}
\begin{vmatrix}
a_1  & & & & & &b_1\vert 0\\
b_2\vert 0  &a_2  & & & & &\\
&\ddots &\ddots & & & &\\
&  &\ddots &1 & & & \\
& & &0 &\ddots & & \\
& & &  &\ddots &a_{n-1} & \\
& & &  & &b_n\vert 0 &a_n\\
\end{vmatrix}
=(-1)^n\prod_{i\neq j}a_i^2\neq 0.
$$

We now consider possible cases for non-zero products of minors (\ref{eq:prodMinors}). Note that $\overline\Gamma_l[\{1,\ldots, n\}\vert \{1,\ldots, n\}]$ equals
$$\prod_{i=1}^n a_i +(-1)^{n+1}\prod_{i=1}^n b_i=(-1)^n\det\Gamma$$ if the CST contains only sequestration reactions (i.e. $s=n$), and $\prod_{i=1}^n a_i$ otherwise.
Therefore the sign of $\overline\Gamma[\{1,\ldots, n\}\vert \{1,\ldots, n\}]\overline\Gamma_l[\{1,\ldots, n\}\vert \{1,\ldots, n\}]$ equals $(-1)^n$ if $s=n$ and equals the sign of 
$\det\Gamma=(-1)^n\prod_{i=1}^n a_i -(-1)^{n+s}\prod_{i=1}^n b_i=(-1)^n[\prod_{i=1}^n a_i-(-1)^s\prod_{i=1}^n b_i]$ otherwise.

If, on the other hand,
$\beta\neq\{1,\ldots,n\}$ then
$\overline\Gamma[\{1,\ldots,n\}\vert \beta]=\det (\Gamma(\{1,\ldots,n\}\vert \beta_1)\vert (-I)(\{1,\ldots,n\}\vert \beta_2))$, where $\vert \beta_1\vert +\vert \beta_2\vert =n$, and $\beta_2$ is nonempty. 
In the Laplace expansion of the determinant $\overline\Gamma[\{1,\ldots,n\}\vert \beta]$ along its last $\vert \beta_2\vert $ columns, we note that there is only one nonzero minor of $(-I)(\{1,\ldots,n\}\vert \beta_2)$, namely $(-I)[\beta_2\vert \beta_2].$ Then, with $\beta'_2$ denoting the complement of $\beta_2$ in $\{1,\ldots, n\}$ we have 
$$\overline\Gamma[\{1,\ldots,n\}\vert \beta]=\epsilon \Gamma[\beta'_2\vert \beta_1] (-I)[\beta_2\vert \beta_2]=\epsilon (-1)^{\vert \beta_2\vert } \Gamma[\beta'_2\vert \beta_1],$$ 
where $\epsilon$ is the signature of the permutation corresponding to the pair ($\alpha_1,\beta_2$). With the same calculation for $\overline\Gamma_l[\beta'_2\vert \beta_1]$ we obtain

\begin{equation}\label{eq:prodmon}
\overline\Gamma[\{1,\ldots,n\}\vert \beta]\overline\Gamma_l[\{1,\ldots,n\}\vert \beta]=(-1)^{\vert \beta_2\vert }\Gamma[\beta'_2\vert \beta_1]\Gamma_l[\beta'_2\vert \beta_1].
\end{equation}

By Lemma \ref{lem:correspProd} the sign of any nonzero product (\ref{eq:prodmon}) is 
$(-1)^{\vert \beta_2\vert +\vert \beta_1\vert }=(-1)^n$ and we conclude that the possible nonzero signs of the product of minors (\ref{eq:prodMinors}) are $(-1)^n$ and $\mbox{sign}(-1)^n(\prod_{i=1}^n a_i -(-1)^s\prod_{i=1}^n b_i)$, and that at least one such product is nonzero. It follows from Theorem  \ref{thm:inj} part 2 that the open CST with outflows is injective if and only if $n=s$ or $\prod_{i=1}^n a_i -(-1)^s\prod_{i=1}^n b_i\ge 0$, i.e. if either $n=s$, $s$ is odd, or if $\prod_{i=1}^n a_i \ge \prod_{i=1}^n b_i.$
\end{proof}

\section{Fully open mass action CSTs} In this section we focus on mass-action kinetics. Suppose the CST has $s$ sequestration reactions. Theorem \ref{thm:CSToutflows} implies that if $n=s$, $s$ is odd, or  if  $\prod_{i=1}^n a_i \ge \prod_{i=1}^n b_i$, then the fully open CST is injective and has at most one positive equilibrium.  In general, non-injectivity of a reaction network does not imply the existence of multiple positive equilibria. However, we now show that non-injective fully open CSTs are in fact multistationary, except for the special case of linear dynamics. Consider a non-injective fully open CST, i.e. suppose $s<n$, $s$ is even and $\prod_{i=1}^n a_i < \prod_{i=1}^n b_i$. Let $\Gamma$ and $\Gamma_l$ denote the stoichiometric and source matrices of the CST without inflow and outflow reactions, and  $\overline\Gamma=[\Gamma\vert  -I]$ and $\overline\Gamma_l=[\Gamma_l\vert  I]$ denote the stoichiometric and source matrices of the open CST with all outflow reactions added (see Remark~\ref{rem:block}).

Together with Theorem \ref{thm:CSToutflows}, the following result completes the characterization of multistationarity in fully open CST networks with mass action kinetics.

\begin{theorem}\label{thm:main} Suppose $s<n$,  $s$ is even, and  $\prod_{i=1}^n a_i < \prod_{i=1}^n b_i$, i.e. the fully open CST network is not injective.
\begin{enumerate}
\item If $s>0$, then the fully open CST network admits multiple positive equilibria under mass action kinetics.

\item If $s=0$ and $1<\prod_{i=1}^n a_i$, then the fully open CST network admits multiple nondegenerate positive equilibria under mass action
 kinetics.
 
\item If $s=0$ and $a_i=1$ for all $i$ then the fully open CST network does not admit multiple positive equilibria under mass action kinetics.
\end{enumerate}
\end{theorem}

\begin{proof} We prove part 1. Without loss of generality, assume that the last reaction is a sequestration, $a_nX_n+b_1X_1\to 0$.
Let $\epsilon>0$ and let $d_i=\frac{b_1\ldots b_i}{a_1\ldots a_i}$ for $1\le i\le n$. Let $D_1=\mbox{diag} (d_1,\ldots, d_n)$ and
$$D=
\begin{bmatrix}
D_1 &0\\
0 &\epsilon I
\end{bmatrix}
.$$

We have 

\begin{eqnarray*}
\overline\Gamma D{\bf 1}&=&[\Gamma D_1\vert -\epsilon I]{\bf 1}< \Gamma D_1 {\bf 1}\le\\
&\le& [-d_1a_1 - d_nb_1, -d_2a_2+d_1b_2,  -d_3a_3+d_2b_3,\ldots, -d_na_n+d_{n-1}b_n]\le 0
\end{eqnarray*}
\noindent (all components except the first one are equal to zero).

Moreover, since $s<n$, $\det\Gamma_l$ contains only one positive monomial. We have $(-1)^n\det(\overline\Gamma D\overline\Gamma_l^T)=(-1)^n\det(\Gamma D_1 \Gamma_l^T-\epsilon I)$, and we see that
$\lim_{\epsilon\to 0}(-1)^n\det(\overline\Gamma D\overline\Gamma_l^T)=
(-1)^n\det\Gamma\det D_1 \det\Gamma_l=(\prod_{i=1}^n a_i -(-1)^s\prod_{i=1}^n b_i)\det D_1 \det\Gamma_l=(\prod_{i=1}^n a_i -\prod_{i=1}^n b_i)\det D_1 \det\Gamma_l<0.$
We can therefore pick $\epsilon>0$ small so that the hypotheses of Theorem \ref{thm:craciunSuff} are satisfied, and the conclusion follows.

For part 2 assume without loss of generality that $a_1>1$.  We show that the CST with only the $X_1$ inflow and outflow reactions added (and not the ones for $X_2,\ldots, X_n$) has nondegenerate multiple equilibria. Once this is done, it follows that the fully open CST multiple nondegenerate steady states as well by Theorem \ref{thm:inheritance}.

Consider then the CST system with inflow/outflow added for $X_1$ only. At steady state,
$\dot x_i=k_{i-1}b_ix_{i-1}^{a_{i-1}}-k_ia_i{x_i}^{a_i}=0$ for $i\ge 2$, which gives
\begin{equation}\label{eq:nice}
x_n^{a_n}=\frac{k_{n-1}}{k_n}\cdot\frac{b_n}{a_n}x_{n-1}^{a_{n-1}}=
\frac{k_{n-1}k_{n-2}}{k_nk_{n-1}}\cdot\frac{b_nb_{n-1}}{a_na_{n-1}}x_{n-2}^{a_{n-2}}=\ldots
=\frac{k_1}{k_n}\frac{b_2\ldots b_n}{a_2\ldots a_n}x_1^{a_1}
\end{equation}
Next, $\dot x_1=0$ yields
$$-k_1a_1x_1^{a_1}+k_nb_1x_n^{a_n}-l_1x_1+f_1=0,$$
where $l_1$ and $f_1$ denote the rate constants of the outflow and inflow of $X_1$. Using (\ref{eq:nice}) we get
$$P(x_1)=k_1a_1\left(\frac{b_1\ldots b_n}{a_1\ldots a_n}-1\right)x_1^{a_1}-l_1x_1+f_1=0.$$
Letting
$$k_1=a_1^{-1}\left(\frac{b_1\ldots b_n}{a_1\ldots a_n}-1\right)^{-1}>0,\ l_1=2^{a_1}-1>0,\ f_1=2^{a_1}-2>0,$$
$P$ has roots 1 and 2. Using (\ref{eq:nice}) we compute two  positive  equilibria with $x_1=1$ and $x_1=2$ respectively. To see that these are nondegenerate equilibria note that the Jacobian matrix of the CST with flow reactions for $X_1$ can be written as (see (\ref{eq:Jac})) 
$$Df=\overline \Gamma D_{\overline v(x)}\overline \Gamma_l^TD_{1/x}$$
which is non-singular if and only if 
$\det(-\overline \Gamma D_{\overline v(x)}\overline \Gamma_l^T)\neq 0$, or equivalently
\begin{equation}\label{eq:nondeg1}
\det(-\Gamma D_{v(x)}\Gamma_l^T)+l_1
(-\Gamma D_{v(x)}\Gamma_l^T)[\{2:n\}\vert \{2:n\}]\neq 0
\end{equation}

We apply Cauchy-Binet to the second determinant:

$$(-\Gamma D_{v(x)}\Gamma_l^T)[\{2:n\}\vert \{2:n\}]=\sum_{\substack{\alpha\subseteq\{1,\ldots, n\}\\ \vert \alpha\vert =n-1}}(-\Gamma[\{2:n\}\vert \alpha])D_{v(x)}[\alpha\vert \alpha]\Gamma_l[\{2:n\}\vert \alpha]$$
and note that  only one minor $\Gamma_l[\{2:n\}\vert \alpha]$ is non-zero, namely when $\alpha=\{2:n\}$. Therefore (\ref{eq:nondeg1}) implies that the condition for nondegeneracy of an equilibrium $x\in{\mathbb R_{>0}^n}$ is 

\begin{eqnarray*}
&&\det(-\Gamma)\det D_{v(x)}\det(\Gamma_l)-\\
-&&l_1(-\Gamma[\{2:n\}\vert \{2:n\}] D_{v(x)}[\{2:n\}\vert \{2:n\}]\Gamma_l[\{2:n\}\vert \{2:n\}]\neq 0,
\end{eqnarray*}
or equivalently
\begin{eqnarray}\nonumber
& &\left(\prod_{i=1}^n a_i - \prod_{i=1}^n b_i\right)\prod_{i=1}^n v_i(x)\prod_{i=1}^n a_i+l_1 \prod_{i=2}^n a_i^2\prod_{i=2}^n v_i(x)\\\nonumber
&=&\prod_{i=2}^n a_i^2 \prod_{i=2}^n v_i(x)\left(\left(1-\frac{b_1\ldots b_n}{a_1\ldots a_n}\right)a_1^2 v_1(x) +l_1\right)\\\label{eq:nondeg2}
&=&\prod_{i=2}^n a_i^2 \prod_{i=2}^n v_i(x)\left(-a_1v_1(x)/k_1 +l_1\right)\neq 0.
\end{eqnarray}

It remains to check (\ref{eq:nondeg2}) for our two equilibrium points. 
For the equilibrium with $x_1=1$ we get $v_1(x)=k_1$ and 
$-a_1+l_1=2^{a_1}-a_1-1>0$ since $a_1\geq 2$. For the equilibrium with $x_1=2$ we get $v_1(x)=k_12^{a_1}$ and 
$-a_12^{a_1}+l_1=(1-a_1)2^{a_1}-1<0$.

For part 3, suppose there are two distinct positive equilibria $x,y\in\mathbb R_{>0}^n$. Let $k_i$ denote the reaction rate of $X_i\to b_{i+1}X_{i+1}$, let $l_i$ denote the outflow rate of $X_i$, and let $f_i$ denote the inflow rate of $X_i$. Set $D_1=\mbox{diag}(k_1,\ldots, k_n)$, $D_2=\mbox{diag}(l_1,\ldots, l_n)$, and $f=[f_1,\ldots, f_n]$. We have $(\Gamma D_1-D_2)x=(\Gamma D_1-D_2)y=-f$. This implies that $\det(\Gamma D_1-D_2)=0$. Writing $x=D{\bf 1}$, where $D=\mbox{diag}(x)$, we have
$(\Gamma D_1-D_2)D{\bf 1}<0$. With $\tilde D_1=D_1D=\mbox{diag}(\tilde k_i)$ and $\tilde D_2=D_2D=\mbox{diag}(\tilde l_i)$ we therefore have
$(\Gamma \tilde D_1-\tilde D_2){\bf 1}<0,$ or equivalently
\begin{equation}\label{eq:unu1}
\tilde k_ib_{i+1}<\tilde k_{i+1}+\tilde l_{i+1}, \ i=1,\ldots, n.
\end{equation}

However, $\det(\Gamma\tilde D_1-\tilde D_2)=\det(\Gamma D_1-D_2)\det D=0$, so that (by Lemma \ref{lem:CSTminors})

$$\prod_{i=1}^n (\tilde k_i+\tilde l_i) = \prod_{i=1}^n \tilde k_i b_i,$$
which contradicts (\ref{eq:unu1}).
\end{proof}

\section{A simple algorithm for deciding the capacity for multistationarity in fully open CST networks}
We present here an algorithm for conclusively establishing the capacity for multistationarity (or lack thereof) for any fully open CST network. 
The algorithm takes the form of a flowchart (see Figure \ref{fig:main})  and also serves as a graphical summary of the main results in the previous section, i.e., Theorems  \ref{thm:CSToutflows} and \ref{thm:main}. (Recall that an injective network does {\em not} have capacity for multistationarity.)

Remarkably, the input to the algorithm  consists of only four integer parameters, two of which are counts of each type of reaction: 
\begin{enumerate}
    \item the number of sequestration reactions ($s$),
    \item the number of transmutation reactions ($t$),
\end{enumerate}
and  the other two are simple functions of the stoichiometric coefficients:

\begin{enumerate}
    \item[3.] $\displaystyle {\rm sgn}\left(\prod a_i - \prod b_i\right)$, 
    \item[4.] $\displaystyle {\rm sgn}\left(\prod a_i - 1\right)$.
\end{enumerate}

Note, in particular, that the conditions above {\em do not depend on the order} in which the sequestration  and transmutation reactions appear along the CST network cycle. 

\begin{figure}[ht]
\hskip1cm
\begin{tikzpicture}[sibling distance=10em,
 every node/.style = {shape=rectangle, rounded corners,
    draw, align=center,
    top color=white, bottom color=gray!20}]]
 \node {$CST(s,t)$}
    child {node {$s$ odd \\ (injective)} }
    child {node {$s$ even}
     		child { node {$s=n$ \\ (injective)}
            }
            child {node {$s<n$}
        		child { node {$\prod a_i \ge \prod b_i$ \\ (injective)}}
        		child { node {$\prod a_i < \prod b_i$ \\ (not injective)}
        			child {node {$s \ne 0$ \\ (multistationary)}}
        			child {node {$s=0$}
				child {node {$a_i=1$ for all $i$ \\ (not multistationary)}}
				child {node {$a_i>1$ for some $i$ \\ (multistationary)}}
	}
	}
        }
        };
\end{tikzpicture}
\caption{Characterization of multistationarity for fully open CST networks. Here $s$ and $t$ represent the number of sequestration and transmutation reactions, respectively.}\label{fig:main}
\end{figure}
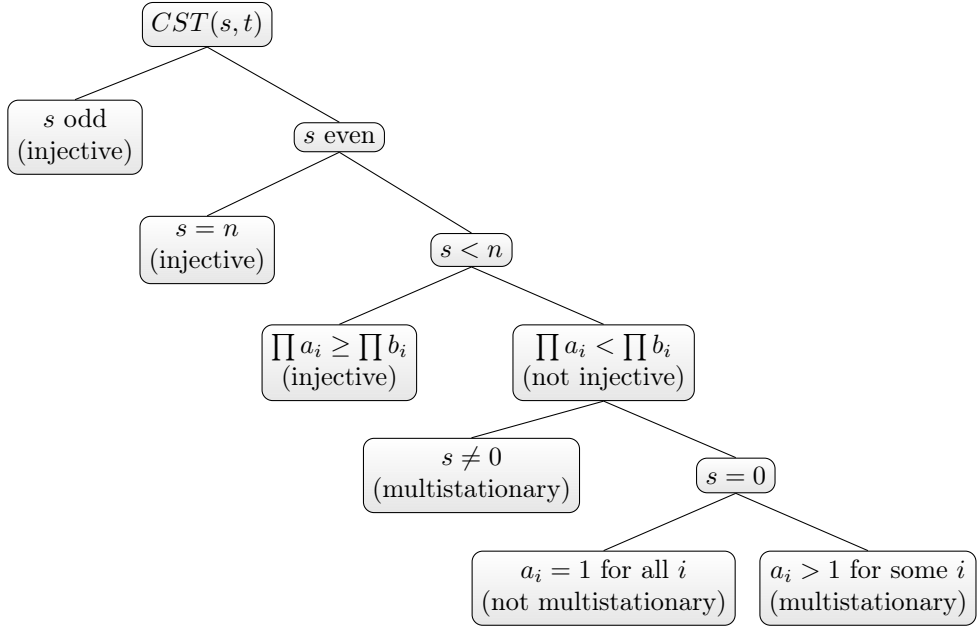

\section {CST atoms of multistationarity} 
We conclude by noting that  fully open CSTs described in Theorem 5.1 parts 1 and 2 are minimally multistationary (i.e., they are  ``atoms of multistationarity'') in the following sense. 

\begin{theorem}\label{thm:atoms}
Let $\cal R$ be a fully open CST network. Suppose we obtain ${\cal R}'$ from $\cal R$ by removing (1) any number of species from all reactions in which they participate and (2) any number of non-flow reactions.  If a trivial reaction (one in which the reactant and product complexes are the same) is obtained in ${\cal R}'$, then it is removed. Also removed are extra copies of repeated reactions in ${\cal R}'$.  Then $\cal R'$ is injective. 
\end{theorem}

Before we sketch the proof of the theorem, we make a few comments. The process of obtaining ${\cal R}'$ from ${\cal R}$ makes ${\cal R}'$ into an {\em embedded network} of ${\cal R}$. We will not present the detailed terminology here (the reader is referred to \cite[Definition 2.2]{joshi2013atoms}), but we briefly describe an example. Consider the fully open CST network ${\cal R}$ 
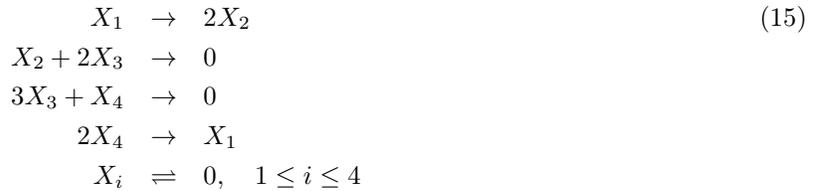
\begin{eqnarray}\label{net:seq1}
X_1&\to& 2X_2\\\nonumber
X_2+2X_3&\to& 0\\\nonumber
3X_3+X_4&\to& 0\\\nonumber
2X_4&\to& X_1\\\nonumber
X_i&\rightleftharpoons& 0,\quad  1\le i\le 4
\end{eqnarray}

Removing the second reaction and species $X_4$ one obtains the network embedded in ${\cal R}$

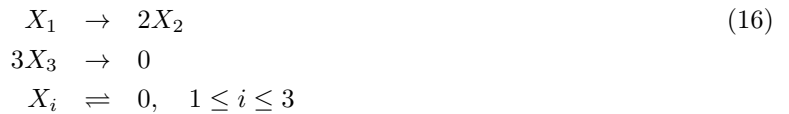
\begin{eqnarray}\label{net:seq2}
X_1&\to& 2X_2\\\nonumber
3X_3&\to& 0\\\nonumber
X_i&\rightleftharpoons& 0,\quad  1\le i\le 3
\end{eqnarray}

Note that the fourth reaction of ${\cal R}$ becomes a duplicate of the inflow reaction $0\to X_1$, which by convention we only list once. 

Theorem \ref{thm:atoms} states that fully open CST networks are minimally multistationary with respect to the ``embedded'' relationship. In other words, they are a version of ``atoms of multistationarity'', a notion introduced in \cite{joshi2013atoms}. We remark that in that paper the authors require that the multiple steady states of an atom also be nondegenerate; the nondegeneracy of steady states for fully open CSTs is subject of future work. 

\medskip

\noindent {\em Proof of Theorem \ref{thm:atoms}.} The proof uses the exact same argument from Theorem \ref{thm:CSTinjopen}. We briefly sketch the argument in what follows. Let $\Gamma$ and $\Gamma_l$ denote the stoichiometric and reactant matrices of the non-flow part of the CST network $\cal R$ and let $\Gamma'$ and $\Gamma'_l$ denote the stoichiometric and reactant matrices of the non-flow part of the embedded network ${\cal R}'$. As usual then, $\bar\Gamma'$ and $\bar\Gamma'_l$ will denote the stoichiometric and reactant matrices of ${\cal R}'$. Minors of $\bar\Gamma'$ and $\bar\Gamma'_l$ are computed using Laplace expansion as in the proof of Theorem \ref{thm:CSToutflows}. Products of corresponding minors of $\bar\Gamma'$ and $\bar\Gamma'_l$ reduce to products of corresponding minors in $\Gamma'$ and $\Gamma'_l$, see (\ref{eq:prodmon}). On the other hand, minors of $\Gamma'$ and $\Gamma'_l$ are strict minors of $\Gamma$ and $\Gamma_l$, which by Lemma \ref{lem:CSTminors} are monomials, or zero. Then the argument after (\ref{eq:prodmon}) follows through.

\section{Application: multistationarity in a model of VEGFR dimerization}\label{sec:Appl}

Endothelial cells make up the lining of blood and lymphatic vessels, and plays important roles in many physiological mechanisms including regulation of vasomotor tone and blood fluidity, control of nutrients and leukocytes across the vascular wall, innate and acquired immunity, and angiogenesis (the growth of new blood vessels from existing vasculature). This diversity of roles of the endothelium is reflected in a remarkable structural and functional heterogeneity of endothelial cells, which can be related to the multistationarity of a pathway induced by VEGF (vascular endothelial  growth factor), a key component of endothelial cell proliferation and angiogenesis \cite{Regan2012.aa}. 

VEGF binds VEGF receptors (VEGFR) found on the surface of the cell via two binding sites, and the binding of a VEGF molecule to two VEGFR molecules induces signal transduction. A standard  model for this dimerization, considered in \cite{VEGF}, is described in Figure \ref{fig:vegf}. We assume that each species has nonzero inflow and outflow rates, i.e. that the network is fully open. This is biologically relevant when, for example, there is an outside domain with high molecule concentration \cite{Ye.2013aa}.  
\begin{figure}[h]
\begin{minipage}{0.49\textwidth}
\begin{tikzpicture}[scale=3,
arc/.style={<->, shorten <=1pt,shorten >=1pt, >=stealth',semithick}]
\node at (0,0) (R) {$R$};
\node at (1,0) (VR) {$VR$}
edge [arc] (R);
\node at (2,0) (R*VR*) {$R^*VR^*$}
edge [arc] (VR);
\node at (0,1) (RR) {$RR$}
edge [arc] (R);
\node at (1,1) (VRR) {$VRR$}
edge [arc] (RR)
edge [arc] (VR);
\node at (2,1) (D) {$R^*VR^{*\Delta}$}
edge [arc] (VRR)
edge [arc] (R*VR*);
\node at (-0.3, 0.3) (R1) {$R$}
edge[arc, <-, out=0, in=-90, >=stealth',semithick] (0,0.5);
\node at (0.3, 1.3) (V1) {$V$}
edge[arc, <-, out=-90, in=180, >=stealth',semithick] (0.5,1);
\node at (0.3, 0.3) (V2) {$V$}
edge[arc, <-, out=-90, in=180, >=stealth',semithick] (0.5,0);
\node at (0.7, 0.3) (R2) {$R$}
edge[arc, <-, out=0, in=-90, >=stealth',semithick] (1,0.5);
\node at (1.3, 0.3) (R3) {$R$}
edge[arc, <-, out=-90, in=180, >=stealth',semithick] (1.5,0);
\end{tikzpicture}
\end{minipage}
\begin{minipage}{0.49\textwidth}
{}
\vspace{1cm}
{\small
\begin{eqnarray}\nonumber
&2R\rightleftharpoons RR \qquad V+R \rightleftharpoons VR\\\nonumber
&V+RR\rightleftharpoons VRR
\qquad VR+R\rightleftharpoons VRR\\\nonumber &VR+R\rightleftharpoons R^*VR^* \qquad
R^*VR^*\rightleftharpoons R^*VR^{*\Delta}\\\nonumber 
&VRR\rightleftharpoons R^*VR^{*\Delta}\\\nonumber 
&R\rightleftharpoons 0, \quad RR\rightleftharpoons 0, 
\quad V\rightleftharpoons 0, \quad VRR\rightleftharpoons 0\\\nonumber
&VRR\rightleftharpoons 0,\quad R^*VR^{*\Delta}\rightleftharpoons 0, \quad R^*VR^*\rightleftharpoons 0
\end{eqnarray}
}
\end{minipage}
\caption{Dimerization of VEGF receptors.  $R$ and $V$ denote the VEGFR receptor monomer and VEGF respectively. The naming convention of bound molecules reflect binding partners (for example $VRR$ is VEGF bound to one VEGFR monomer of a VEGFR dimer, while $RVR$ denotes VEGF bound to two VEGFR monomers. A phoshorylated receptor is marked with $^*$, and $\Delta$ indicates that there is a bond between any two of the three components of the molecule. See \cite{VEGF} for details.
}\label{fig:vegf}
\end{figure}
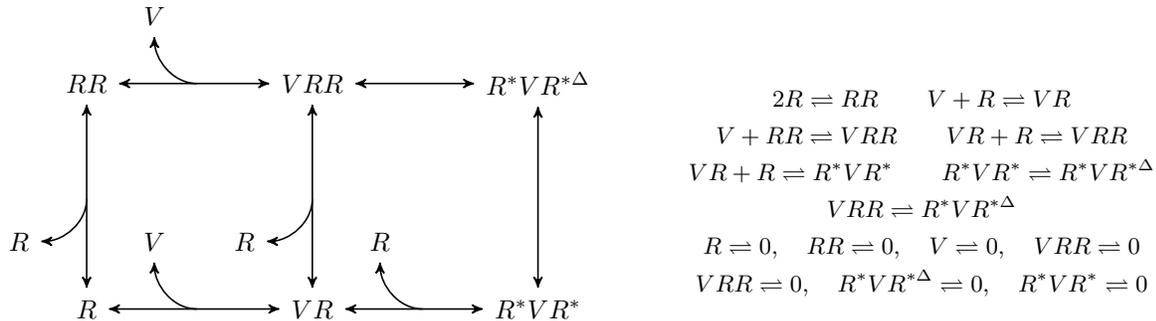

We use the results in this paper to show that the VEGFR dimerization network in Figure \ref{fig:vegf} is multistationary. Namely, we exhibit a multistationary CST that can be built up to the VEGF network by way of modifications in Theorem \ref{thm:inheritance}. An automated way of searching for such a CST substructure in general networks is currently being implemented in CoNtRol \cite{Donnell.2014fg}.

We start with the fully open CST network 

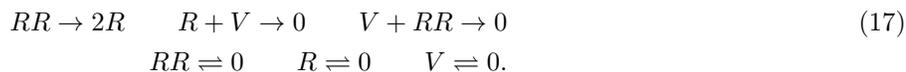
\begin{eqnarray}\label{eq:net1}
    RR\to 2R \qquad R+V\to 0 \qquad V+RR\to 0\\\nonumber
    RR\rightleftharpoons 0\qquad  R\rightleftharpoons 0\qquad  V\rightleftharpoons 0.
\end{eqnarray}
\noindent which is multistationary according to the first part of Theorem \ref{thm:main}  (with $n=3$, $s=2$, $a_1a_2a_3=1$, $b_1b_2b_3=2$). This fully open three-species CST is in fact a sequestration network, and has nondegenerate multiple  positive steady states (see Remark in Section \ref{sec:seqnet}). This allows us to apply Theorem \ref{thm:inheritance} in the following way. At each step below, the network modifications are indicated in bold, and all modified networks have multiple positive nondegenerate steady states.   
First, we add species $VR$ and $VRR$ into the second and third reactions, together with inflow and outflow (Theorem \ref{thm:inheritance} part 2):
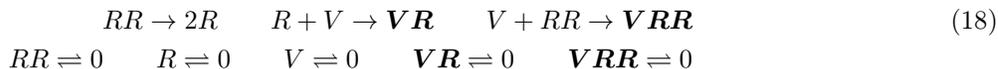
\begin{eqnarray}\label{eq:net2}
   RR\to 2R \qquad R+V\to {\boldsymbol{VR}} \qquad V+RR\to {\boldsymbol{VRR}}\\\nonumber
   RR\rightleftharpoons 0\qquad  R\rightleftharpoons 0\qquad  V\rightleftharpoons 0
   \qquad 
   {\boldsymbol{VR}}\boldsymbol{\rightleftharpoons} 0
    \qquad 
    {\boldsymbol{VRR}}\boldsymbol{\rightleftharpoons} 0
\end{eqnarray}   
Next, we add inflow/outflow reactions for the remaining species 
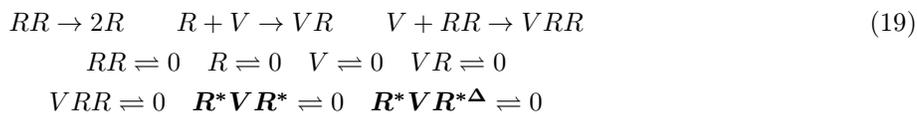
\begin{eqnarray}\label{eq:net3}
    &RR\to 2R \qquad R+V\to VR \qquad V+RR\to VRR\\ \nonumber
    &RR\rightleftharpoons 0\quad R\rightleftharpoons 0\quad  V\rightleftharpoons 0\quad 
    VR\rightleftharpoons 0\\ \nonumber
    &VRR\rightleftharpoons 0 \quad
    \boldsymbol{R^*VR^*}\boldsymbol{\rightleftharpoons} 0\quad  \boldsymbol{R^*VR^{*\Delta}}\boldsymbol{\rightleftharpoons} 0
\end{eqnarray}  
Since all seven species have inflows and outflow reactions,  the stoichiometric subspace of (\ref{eq:net3}) is ${\mathbb R}^7$, and it contains any reaction among the species of (\ref{eq:net3}). Theorem \ref{thm:inheritance} part 3 implies that adding the remaining reactions (including the reverse of the first three reactions in (\ref{eq:net3})) preserves multistationarity:

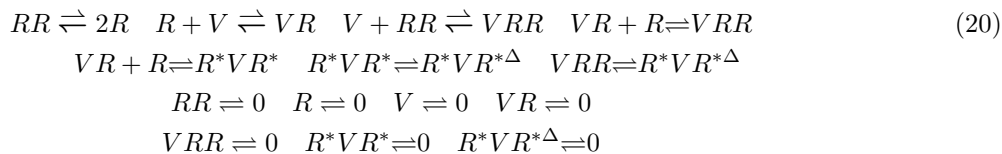
\begin{eqnarray}\label{eq:net4}
    &RR \oset{\rightharpoonup}{\boldsymbol{\leftharpoondown}}2R \quad 
 R+V\oset{\rightharpoonup}{\boldsymbol{\leftharpoondown}} VR \quad V+RR\oset{\rightharpoonup}{\boldsymbol{\leftharpoondown}}   VRR\quad VR+R{\boldsymbol{\rightleftharpoons}} VRR\\\nonumber
     &\qquad VR+R{\boldsymbol\rightleftharpoons} R^*VR^* \quad
     R^*VR^*{\boldsymbol\rightleftharpoons} R^*VR^{*\Delta}\quad VRR{\boldsymbol\rightleftharpoons} R^*VR^{*\Delta}\\\nonumber
     &RR\rightleftharpoons 0\quad R\rightleftharpoons 0\quad  V\rightleftharpoons 0\quad 
     VR\rightleftharpoons 0\\ \nonumber
     &VRR\rightleftharpoons 0 \quad
   {R^*VR^*}{\rightleftharpoons} 0\quad  {R^*VR^{*\Delta}}{\rightleftharpoons} 0
\end{eqnarray}  

\noindent Therefore the VEGFR dimerization network (\ref{eq:net4}) has multiple nondegenerate positive eqiulibria. 

 \bibliography{gsp}

\begin{thebibliography}{10}

\bibitem{Angeli.2010aa}
David Angeli, Patrick De~Leenheer, and Eduardo Sontag.
\newblock Graph-theoretic characterizations of monotonicity of chemical
  networks in reaction coordinates.
\newblock {\em J Math Biol}, 61(4):581--616, Oct 2010.

\bibitem{angeli2004detection}
David Angeli, James~E Ferrell, and Eduardo~D Sontag.
\newblock Detection of multistability, bifurcations, and hysteresis in a large
  class of biological positive-feedback systems.
\newblock {\em Proceedings of the National Academy of Sciences},
  101(7):1822--1827, 2004.

\bibitem{DSRcyles}
Xue Bai, Casian Pantea, Lorand Parajdi, Galyna Voitiuk, and Polly~Y. Yu.
\newblock Cycles in mass-action networks and multistationarity.
\newblock {\em in preparation}.

\bibitem{Banaji.2007aa}
M~Banaji, P~Donnell, and S~Baigent.
\newblock P matrix properties, injectivity, and stability in chemical reaction
  systems.
\newblock {\em SIAM J. Appl. Math}, 67(6):1523--1547, 2007.

\bibitem{Banaji.2009aa}
Murad Banaji and Gheorghe Craciun.
\newblock {Graph-theoretic approaches to injectivity and multiple equilibria in
  systems of interacting elements}.
\newblock {\em Communications in Mathematical Sciences}, 7(4):867 -- 900, 2009.

\bibitem{banaji2010graph}
Murad Banaji and Gheorghe Craciun.
\newblock Graph-theoretic criteria for injectivity and unique equilibria in
  general chemical reaction systems.
\newblock {\em Advances in Applied Mathematics}, 44(2):168--184, 2010.

\bibitem{banaji2016some}
Murad Banaji and Casian Pantea.
\newblock Some results on injectivity and multistationarity in chemical
  reaction networks.
\newblock {\em SIAM Journal on Applied Dynamical Systems}, 15(2):807--869,
  2016.

\bibitem{banaji2018inheritance}
Murad Banaji and Casian Pantea.
\newblock The inheritance of nondegenerate multistationarity in chemical
  reaction networks.
\newblock {\em SIAM Journal on Applied Mathematics}, 78(2):1105--1130, 2018.

\bibitem{Ye.2013aa}
Ye~Chen, Christopher Short, Ad{\'a}m~M Hal{\'a}sz, and Jeremy~S Edwards.
\newblock The impact of high density receptor clusters on {V}{E}{G}{F}
  signaling.
\newblock {\em Electronic proceedings in theoretical computer science},
  2013:37, 2013.

\bibitem{conradi2017graph}
Carsten Conradi and Maya Mincheva.
\newblock Graph-theoretic analysis of multistationarity using degree theory.
\newblock {\em Mathematics and Computers in Simulation}, 133:76--90, 2017.

\bibitem{Craciun.2019df}
Gheorghe Craciun.
\newblock Polynomial dynamical systems, reaction networks, and toric
  differential inclusions.
\newblock {\em SIAM Journal on Applied Algebra and Geometry}, 3:87--106, 01
  2019.

\bibitem{Craciun.2009aa}
Gheorghe Craciun, A~Dickenstein, A~Shiu, and Bernd Sturmfels.
\newblock Toric dynamical systems.
\newblock {\em Journal of Symbolic Computation}, 44(11):1551--1565, 2009.

\bibitem{craciun2005multiple}
Gheorghe Craciun and Martin Feinberg.
\newblock Multiple equilibria in complex chemical reaction networks: {I.} {T}he
  injectivity property.
\newblock {\em SIAM Journal on Applied Mathematics}, 65(5):1526--1546, 2005.

\bibitem{craciun2006multiple}
Gheorghe Craciun and Martin Feinberg.
\newblock Multiple equilibria in complex chemical reaction networks: extensions
  to entrapped species models.
\newblock {\em IEE Proceedings-Systems Biology}, 153(4):179--186, 2006.

\bibitem{delayPolly}
Gheorghe Craciun, Maya Mincheva, Casian Pantea, and Polly~Y. Yu.
\newblock A graph-theoretic condition for delay stability of reaction systems,
  2021.

\bibitem{Craciun.2011ox}
Gheorghe Craciun, Casian Pantea, and Eduardo Sontag.
\newblock Graph-theoretic analysis of multistability and monotonicity for
  biochemical reaction networks.
\newblock {\em Design and Analysis of Biomolecular Circuits}, pages 63--72, 01
  2011.

\bibitem{craciun2006understanding}
Gheorghe Craciun, Yangzhong Tang, and Martin Feinberg.
\newblock Understanding bistability in complex enzyme-driven reaction networks.
\newblock {\em Proceedings of the National Academy of Sciences},
  103(23):8697--8702, 2006.

\bibitem{Donnell.2014fg}
Pete Donnell, Murad Banaji, Anca Marginean, and Casian Pantea.
\newblock {C}o{N}t{R}ol: an open source framework for the analysis of chemical
  reaction networks.
\newblock {\em Bioinformatics}, 30(11), 2014.

\bibitem{feinberg2019foundations}
Martin Feinberg.
\newblock {\em Foundations of chemical reaction network theory}.
\newblock Springer, Switzerland, 2019.

\bibitem{feliu2013simplifying}
Elisenda Feliu and Carsten Wiuf.
\newblock Simplifying biochemical models with intermediate species.
\newblock {\em Journal of The Royal Society Interface}, 10(87):20130484, 2013.

\bibitem{felix2016analyzing}
Bryan F{\'e}lix, Anne Shiu, and Zev Woodstock.
\newblock Analyzing multistationarity in chemical reaction networks using the
  determinant optimization method.
\newblock {\em Applied Mathematics and Computation}, 287:60--73, 2016.

\bibitem{ferrell2002self}
James~E Ferrell~Jr.
\newblock Self-perpetuating states in signal transduction: positive feedback,
  double-negative feedback and bistability.
\newblock {\em Current opinion in cell biology}, 14(2):140--148, 2002.

\bibitem{Horn.1972ab}
FJM Horn.
\newblock Necessary and sufficient conditions for complex balancing in chemical
  kinetics.
\newblock {\em Archive for Rational Mechanics and Analysis}, 49(3):172--186,
  1972.

\bibitem{Horn.1972aa}
FJM Horn and R~Jackson.
\newblock General mass action kinetics.
\newblock {\em Archive for Rational Mechanics and Analysis}, 47(2):81--116,
  1972.

\bibitem{joshi2013complete}
Badal Joshi.
\newblock Complete characterization by multistationarity of fully open networks
  with one non-flow reaction.
\newblock {\em Applied Mathematics and Computation}, 219:6931--6945, 2013.

\bibitem{joshi2012simplifying}
Badal Joshi and Anne Shiu.
\newblock {Simplifying the Jacobian Criterion for precluding multistationarity
  in chemical reaction networks}.
\newblock {\em SIAM Journal on Applied Mathematics}, 72(3):857--876, 2012.

\bibitem{joshi2013atoms}
Badal Joshi and Anne Shiu.
\newblock Atoms of multistationarity in chemical reaction networks.
\newblock {\em Journal of Mathematical Chemistry}, 51(1):153--178, 2013.

\bibitem{joshi2015survey}
Badal Joshi and Anne Shiu.
\newblock A survey of methods for deciding whether a reaction network is
  multistationary.
\newblock {\em ``Chemical Dynamics" -- special issue of Mathematical Modelling
  of Natural Phenomena}, 10(5):47--67, 2015.

\bibitem{joshi2017small}
Badal Joshi and Anne Shiu.
\newblock Which small reaction networks are multistationary?
\newblock {\em SIAM Journal on Applied Dynamical Systems}, 16(2):802--833,
  2017.

\bibitem{VEGF}
Feilim Mac~Gabhann and Aleksander Popel.
\newblock Dimerization of {V}{E}{G}{F} receptors and implications for signal
  transduction: A computational study.
\newblock {\em Biophysical chemistry}, 128:125--39, 08 2007.

\bibitem{mincheva2007graph}
Maya Mincheva and Marc~R Roussel.
\newblock Graph-theoretic methods for the analysis of chemical and biochemical
  networks. i. multistability and oscillations in ordinary differential
  equation models.
\newblock {\em Journal of Mathematical Biology}, 55(1):61--86, 2007.

\bibitem{Mueller.2016aa}
Stefan M{\"u}ller, Elisenda Feliu, Georg Regensburger, Carsten Conradi, Anne
  Shiu, and Alicia Dickenstein.
\newblock Sign conditions for injectivity of generalized polynomial maps with
  applications to chemical reaction networks and real algebraic geometry.
\newblock {\em Foundations of Computational Mathematics}, 16(1):69--97, 2016.

\bibitem{Pantea.2012ss}
Casian Pantea.
\newblock On the persistence and global stability of mass-action systems.
\newblock {\em SIAM Journal on Mathematical Analysis}, 44(3):1636--1673, 2012.

\bibitem{Regan2012.aa}
E.R. Regan and W.C. Aird.
\newblock Dynamical systems approach to endothelial heterogeneity.
\newblock {\em Circ. Res.}, 111(1):110--130, 2012.

\bibitem{Shinar.2012aa}
Guy Shinar and Martin Feinberg.
\newblock Concordant chemical reaction networks.
\newblock {\em Mathematical Biosciences}, 240:92--113, 2012.

\bibitem{tang2019bistability}
Xiaoxian Tang and Jie Wang.
\newblock Bistability of sequestration networks.
\newblock {\em Discrete \& Continuous Dynamical Systems - B}, 26(3):1337--1357,
  2021.

\bibitem{wiuf2013power}
Carsten Wiuf and Elisenda Feliu.
\newblock Power-law kinetics and determinant criteria for the preclusion of
  multistationarity in networks of interacting species.
\newblock {\em SIAM Journal on Applied Dynamical Systems}, 12(4):1685--1721,
  2013.

\bibitem{yu2018mathematical}
Polly~Y Yu and Gheorghe Craciun.
\newblock Mathematical analysis of chemical reaction systems.
\newblock {\em Israel Journal of Chemistry}, 58(6-7):733--741, 2018.

\end{thebibliography}
 \bibliographystyle{plain}
\end{document}